\documentclass{amsart}
\usepackage{amsfonts,amssymb,amscd,amsmath,enumerate,verbatim,calc}
\usepackage[all]{xy}

\newcommand{\CM}{Cohen-Macaulay}

\newcommand{\wrt}{with respect to}

\newcommand{\n}{\mathfrak{n} }
\newcommand{\m}{\mathfrak{m} }
\newcommand{\M}{\mathfrak{M} }

\newcommand{\Z}{\mathbb{Z} }
\newcommand{\F}{\mathcal{F} }
\newcommand{\G}{\mathcal{G}}
\newcommand{\Hc}{\mathcal{H}}
\newcommand{\rt}{\rightarrow}

\newcommand{\ov}{\overline}

\newcommand{\xb}{\mathbf{x}}

\newcommand{\wh}{\widehat }

\newcommand{\depth}{\operatorname{depth}}
\newcommand{\ord}{\operatorname{ord}}

\newcommand{\soc}{\operatorname{soc}}
\newcommand{\type}{\operatorname{type}}

\newcommand{\codim}{\operatorname{codim}}
\newcommand{\Hom}{\operatorname{Hom}}
\newcommand{\Ext}{\operatorname{Ext}}

\theoremstyle{plain}

\newtheorem{theorem}{Theorem}[section]

\newtheorem{proposition}[theorem]{Proposition}

\theoremstyle{definition}

\newtheorem{s}[theorem]{}
\newtheorem{remark}[theorem]{Remark}
\newtheorem{example}[theorem]{Example}

\theoremstyle{remark}

\begin{document}

\title[Teter rings]{Higher dimensional Teter rings}
\author{Tony~J.~Puthenpurakal}
\date{\today}
\address{Department of Mathematics, IIT Bombay, Powai, Mumbai 400 076, India}

\email{tputhen@gmail.com}
\subjclass{Primary 13H10 ; Secondary 13A30 }
\keywords{ Gorenstein rings, Finite Cohen–Macaulay type, canonical module, associated graded rings and modules}
 \begin{abstract}
Let $(A,\m)$ be a complete \CM \ local ring. Assume $A$ is not Gorenstein. We say $A$ is a Teter ring if there exists a complete Gorenstein ring $(B,\n)$ with $\dim B = \dim A$ and a surjective map
$B \rt A$ with $e(B) - e(A) = 1$ (here $e(A)$ denotes multiplicity of $A$). We give an intrinsic characterization of Teter rings which are domains. We say a Teter ring is a strongly Teter ring if $G(B) = \bigoplus_{i \geq 0}\n^i/\n^{i+1}$ is also a Gorenstein ring. We give an intrinsic characterizations of strongly Teter rings which are domains. If $k$ is algebraically closed field of characteristic zero and $R$ is a standard graded \CM \ $k$-algebra of finite representation type (and not Gorenstein) then we show that $\wh{R_\M}$ is a Teter ring (here $\M$ is the maximal homogeneous ideal of $R$).
\end{abstract}
 \maketitle
\section{introduction}
Let $(A,\m)$ be an Artinian local ring.  Note there exists an Artinian Gorenstein local ring $B$ with a surjective map $B \rt A$. Let $\ell(A)$ denote length of $A$. The Gorenstein co-length of $A$ is defined as
$$g(A) = \min \{ \ell(B) - \ell(A) \mid \text{$B$ is  Gorenstein Artin  local ring mapping onto $A$  } \}.$$
We note that $g(A) = 0$ if and only if $A$ is Gorenstein. Observe that $g(A) = 1$ if and only if $A$ is not Gorenstein and $A = B/\soc(B)$ for a Gorenstein Artin local ring $B$. W. Teter gives an intrinsic characterization for such rings in his paper \cite{T}. In their paper \cite{HV}, C. Huneke and A. Vraciu refer to these rings as Teter’s rings.

We generalize the definition of Teter rings to higher dimensions. Let $(A,\m)$ be a local \CM \ local ring of dimension $d$. We assume $A$ is a quotient of a regular local ring $T$. Let $e(A)$ denote the multiplicity of $A$. We note that there exists a  Gorenstein local ring $B$ mapping onto $A$ with $\dim B = \dim A$. Let $U(A)$ denote the set of all  Gorenstein local rings $B$
which are quotients of $T$, mapping onto $A$ with $\dim B = \dim A$.
Set
$$g(A) = \min \{ e(B) - e(A) \mid \text{$B \in U(A)$ } \}.$$

We note that $g(A) = 0$ if and only if $A$ is a Gorenstein local ring. If $g(A) = 1$ then we call $A$ to be a \textit{Teter} ring. If $B \in U(A)$ such that $e(B) - e(A) = 1$ then we call $B$ to be a\textit{ Teter Gorenstein approximation} of $A$. Note we are \emph{not} asserting that $B$ is uniquely determined by $A$.

 In general giving an intrinsic characterization of Teter's ring seems to be difficult. However for domains we have
\begin{theorem}
\label{main} Let $(A,\m)$ be a  \CM \  local domain. Assume $A$ is a quotient of a regular local ring $T$.  Assume $A$ is not Gorenstein. Let $\omega_A$ denote the canonical module of $A$.
The following assertions are equivalent:
\begin{enumerate}[\rm (i)]
\item $A$ is a Teter ring.
\item There exists a proper ideal $J$ of $A$ with $\omega_A \cong J$ and $\codim A/J \leq 1$.
\end{enumerate}
Furthermore if $A$ is a  Teter-ring then the Teter Gorenstein approximation of $A$ \\ is unique (up to isomorphism).
\end{theorem}
 Also if $(R,\m)$ is a local Noetherian ring then $\codim R = \mu(\m) - \dim R$ (here $\mu(\m)$ denotes the minimal number of generators of $\m$).

If $(R,\n)$ is a Noetherian local ring then $G(R) = \bigoplus_{i \geq 0}\n^i/\n^{i+1}$ is the associated graded ring of $R$. We say $A$ is strongly Teter if $A$ has a Teter Gorenstein approximation $(B,\n)$ with $G(B)$ Gorenstein.
Being strongly Gorenstein is a very restricted property. We show
\begin{theorem}
\label{strong}
Let $(A,\m)$ be a  Cohen-Macaulay local  ring. Assume $A$ is a quotient of a regular local ring $T$. If $A$ is strongly Teter then $G(A)$ is \CM.
\end{theorem}
Theorem \ref{strong} can be used to give examples of rings which are Teter but not strongly Teter, see \ref{non-std-graded-2}.
We give an intrinsic characterization of \CM \ local domains which are strongly Teter.
\begin{theorem}
\label{main-strong} Let $(A,\m)$ be a  \CM \  local domain.  Assume $A$ is a quotient of a regular local ring $T$. Assume $A$ is not Gorenstein.
The following assertions are equivalent:
\begin{enumerate}[\rm (i)]
\item $A$ is a strongly Teter ring.
\item $A$ satisfies the following properties.
 \begin{enumerate}[\rm (a)]
 \item $G(A)$ is \CM.
 \item There exists a proper ideal $J$ of $A$ with $\omega_A \cong J$ and $\codim A/J \leq 1$.
 \item Consider the filtration $\F = \{ J_n  = J \cap \m^n \}_{n \geq 0}$ on $J$. Then $G_\F(J)$ is isomorphic to the graded canonical module of $G(A)$ (up to a shift).
 \end{enumerate}
\end{enumerate}
\end{theorem}

In \cite{HV}, the authors stated that rings of utmost countable representation type should be close to Gorenstein rings. We prove this for homogeneous rings of finite representation type.  Recall that a graded algebra over a field has finite representation type if it has up to isomorphisms and shifts in the grading, only a finite number of indecomposable maximal Cohen-Macaulay modules.
\begin{theorem}\label{fin-rep}
Let $k$ be an algebraically closed field of characteristic zero. Let $R$ be a standard \CM \
$k$-algebra with maximal homogeneous ideal $\M$. Assume $R$ is NOT Gorenstein. Further assume $R$ is of finite \CM \ type. Let $A = \wh{R_\M}$ be its completion \wrt \ $\M$. Then $A$ is a Teter ring.
\end{theorem}

We now describe in brief the contents of this paper. In section two we prove Theorem \ref{main}. In section three we give proof of Theorem \ref{strong}. In section four we give a proof of Theorem \ref{main-strong}. Finally in section five we give examples
and prove Theorem \ref{fin-rep}.

\section{Proof of Theorem \ref{main}}
In this section we give a proof of Theorem \ref{main}. We first prove
\begin{proof}[Proof of (i) $\implies$ (ii)]
Let $\dim A = d$.
As $A$ is a Teter ring  there exists a $d$-dimensional  Gorenstein local ring $B$ mapping onto $A$  with $e(B) - e(A) = 1$. We have an exact sequence $$0 \rt M \xrightarrow{g} B \xrightarrow{f} A \rt 0, \quad \text{where $e(M) = 1$.}$$
We note that $M$ is a MCM $B$-module with $e(M) = 1$. Dualizing the above exact sequence we obtain an exact sequence
\[
0 \rt \omega_A \xrightarrow{f^\vee} B \xrightarrow{g^\vee} M^\vee \rt 0.
\]
We note that $M^\vee$ is a \CM \ ring of dimension $d$. As $ e(M^\vee) = e(M) = 1$, it follows that $M^\vee$ is a regular local ring. Denote it by $Q$.
Consider the map $\omega_A \xrightarrow{f\circ f^\vee} A$ and let $K$ be the kernel of this map. Set $V = \omega_A/K$ and let $i^\prime \colon V \rt A$ be the induced injection.
We have the following commutative diagram:
\[
  \xymatrix
{
\
 \
  & 0
\ar@{->}[d]
 & 0
\ar@{->}[d]
& 0
\ar@{->}[d]
& \
\\
0
 \ar@{->}[r]
  & N
\ar@{->}[r]
\ar@{->}[d]
 & M
\ar@{->}[r]
\ar@{->}[d]^{g}
& Q^{\prime \prime}
\ar@{->}[r]
\ar@{->}[d]^{j}
&0
\\
 0
 \ar@{->}[r]
  & \omega_A
\ar@{->}[r]^{f^\vee}
\ar@{->}[d]^{\pi}
 & B
\ar@{->}[r]^{g^\vee}
\ar@{->}[d]^{f}
& Q
\ar@{->}[r]
\ar@{->}[d]
&0
\\
 0
 \ar@{->}[r]
  &V
\ar@{->}[r]^{i^\prime}
\ar@{->}[d]
 & A
\ar@{->}[r]^{p^\prime}
\ar@{->}[d]
& Q^\prime
    \ar@{->}[r]
    \ar@{->}[d]
    &0
\\
\
\
  & 0
\
 & 0
\
& 0
    \
    & \
\
 }
\]
We  first claim that either $N$ or $Q^{\prime \prime}$ is zero. Suppose if possible both of them are non-zero. We note that as associate primes of $\omega_A$ have dimension $d$ it follows that $\dim N = d$.
Similarly $\dim Q^{\prime \prime} = d$. It follows that the modified multiplicity $e(N), e(Q^{\prime\prime})$ are positive. As dimensions of $N, M, Q^{\prime\prime}$ is $d$ we get by \cite[4.7.7]{BH};
$e(N) + e(Q^{\prime\prime}) = e(M) = 1.$ This is a contradiction

Suppose if possible $Q^{\prime \prime} = 0$. Then $Q^\prime = Q$. We note that $Q = M^\vee \neq 0$. We note that $Q$   is a $d$-dimensional  quotient of $A$ which is not $A$ (as $A$ is NOT regular). This is a contradiction as $A$ is a domain.
So $N = 0$.  Thus $V = \omega_A$. We note that $M = Q^\prime $ is a MCM $Q$-module with multiplicity  one. So $Q^{\prime \prime} \cong Q$. Thus $Q^\prime = Q/(a)$ for some $a$. If $Q^\prime = 0$ then note that $\omega_A \cong A$. This forces $A$ to be Gorenstein, a contradiction. So $Q^\prime$ is non-zero and note that $\codim Q^\prime \leq 1$. The result follows.
\end{proof}

We have to prove a slightly more general version of the assertion (ii) $\implies$ (i) of Theorem \ref{main}.
\begin{theorem}
\label{fiber} Let $(A,\m)$ be a \CM \ local ring of dimension $d \geq 1$. Assume $A$ is not Gorenstein. Assume $A$ is a quotient of a regular local ring $T$. Suppose the canonical module $\omega_A$ of $A$ is isomorphic to a proper ideal $J$ of $A$ such that
$\codim A/J \leq 1$. Then $A$ is a Teter ring.
\end{theorem}
\s \label{rem-gen}Before proving the above result we make the following remarks:
\begin{enumerate}
  \item The ring $A/J$ is a hypersurface ring (or regular local). We note that it suffices to show $A/J$ is \CM. To prove this note that $\depth A/J \geq d -1$. As $e(J) = e(\omega) = e(A)$, it follows from \cite[4.7.7]{BH} that $\dim A/J \leq d - 1$. The assertion follows.
  \item Note that $\omega_A  \cong J$ has a rank as $A/J$ is a torsion $A$-module. It follows that $A$ is generically Gorenstein, see \cite[3.3.18]{BH}.
  \item  $A/J = Q/(a)$ for some regular local ring $Q$ of dimension $d$. This follows from the fact that $A/J$ is a  local ring which has utmost hypersurface singularity is a quotient of a regular local ring $T$.
\end{enumerate}
We now give
\begin{proof}[Proof of Theorem \ref{fiber}]
  By \ref{rem-gen} there exists a complete regular local ring $Q$ of dimension $d$ such that $A/J = Q/(a)$ for some non-zero $a$. Let $B$ be the fiber-product of $A \rt Q/(a) = A/J$ and $Q \rt Q/(a)$. It is routine to verify that the induced homomorphisms $p \colon B \rt Q$ and $q \colon B\rt A$ are surjective. Also notice $B$ is local.
  We have the following commutative diagram
  \[
  \xymatrix
{
 0
 \ar@{->}[r]
  & \omega_A
\ar@{->}[r]^{i}
\ar@{->}[d]^{f}
 & B
\ar@{->}[r]^{p}
\ar@{->}[d]^{q}
& Q
\ar@{->}[r]
\ar@{->}[d]^{\pi^\prime}
&0
\\
 0
 \ar@{->}[r]
  & \omega_A
\ar@{->}[r]^{i^\prime}
 & A
\ar@{->}[r]^{\pi}
& Q/(a)
    \ar@{->}[r]
    &0
\
 }
\]
We note that $f$ is an isomorphism, \cite[Chapter III, 5.4]{K}.
As $A, Q$ are quotient of $T$, so $B$ is also a quotient of $T$.
As we have an exact sequence $0  \rt \omega_A \rt B \xrightarrow{p} Q \rt 0$ is exact it follows that $B$ is \CM \ of dimension $d$. We may by using prime avoidance construct a sequence $\xb = x_1, \ldots, x_d$  which is an $B \oplus \omega_A \oplus Q \oplus A$   regular sequence, such that images of $\xb$ generate the maximal ideal of $Q$. Going mod $\xb$ we obtain an exact sequence
$$ 0 \rt \omega_A/\xb \omega A \rt B/\xb B \xrightarrow{\ov{p}} k \rt 0.$$
If $t \in \soc(B/\xb B)$ then $\ov{p(t)} = 0$. It follows that $\soc(B/\xb B) \cong \soc(\omega_A/\xb \omega_A) \cong k$ as $\omega_A/\xb \omega_A $ is the canonical module of $A/\xb A$. Thus $B/\xb B$ is a Gorenstein ring and hence $B$ is a Gorenstein ring. We note that $\ker (B \xrightarrow{q} A) \cong \ker (Q \rt Q/(a)) = Q$. This we have an exact sequence
$0 \rt M \rt B \rt A \rt 0$ where $M \cong Q$. So $e(M) = 1$. Thus $e(B) = e(A) + 1$. It follows that $A$ is a Teter ring.
\end{proof}

We now give
\begin{proof}[Proof of Theorem \ref{main} continued] We already have proved (i) $\implies$ (ii). The assertion (ii) $\implies$ (i) follows from Theorem \ref{fiber}.
We now have to prove uniqueness of the Teter Gorenstein approximation of $A$.
Set $C = A\times_{Q/(a)}Q$ the fiber-product and let $j \colon C \rt A$ and $l \colon C \rt Q$ be the natural maps. By the proof of (ii) $\implies$ (i) we get that $C$ is a Teter Gorenstein approximation of $A$. We prove $B \cong C$.
 By our proof of (i) $\implies$ (ii) we obtain a commutative diagram
  \[
  \xymatrix
{
 0
 \ar@{->}[r]
  & \omega_A
\ar@{->}[r]^{i}
\ar@{->}[d]^{g}
 & B
\ar@{->}[r]^{p}
\ar@{->}[d]^{q}
& Q
\ar@{->}[r]
\ar@{->}[d]^{\pi^\prime}
&0
\\
 0
 \ar@{->}[r]
  & \omega_A
\ar@{->}[r]^{i^\prime}
 & A
\ar@{->}[r]^{\pi}
& Q/(a)
    \ar@{->}[r]
    &0
\
 }
\]
where $g$ is an isomorphism and $p,q$ are surjective.  By the universal property of fiber products we have a ring map $\psi \colon B \rt C$ such that $p = l\circ \psi$ and $ q = j \circ \psi$.
So we have a commutative diagram
  \[
  \xymatrix
{
 0
 \ar@{->}[r]
  & \omega_A
\ar@{->}[r]^{i}
\ar@{->}[d]^{u}
 & B
\ar@{->}[r]^{p}
\ar@{->}[d]^{\psi}
& Q
\ar@{->}[r]
\ar@{->}[d]^{1_Q}
&0
\\
 0
 \ar@{->}[r]
  & \omega_A
\ar@{->}[r]^{i^\prime}
 & C
\ar@{->}[r]^{\pi}
& Q
    \ar@{->}[r]
    &0
\
 }
\]
We note that $g = f \circ u$ where $f \colon \omega_A \rt \omega_A$ is the isomorphism induced by the fiber-product diagram. As $g$ is an isomorphism it follows that $u$ is an isomorphism.
 So $\psi$ is an isomorphism. The result follows.
\end{proof}

\section{Proof of Theorem \ref{strong}}
In this section we give
\begin{proof}[Proof of Theorem \ref{strong}]
Let $(B, \n)$ be a Teter Gorenstein approximation of $A$. So we have an exact sequence
$0 \rt M \rt B \rt A \rt 1$ where $M$ is a MCM $B$-module of multiplicity one. We note that we may assume that the residue field of $B$ is infinite (otherwise we may consider $B^\prime = B[X]_{\n B[X]}$ and $A^\prime = A\otimes_B B^\prime$).

We induct on dimension $d$ of $A$. When $d = 0$ we have nothing to prove.

The case when $d = 1$ is the most important case. Suppose if possible $G(A)$ is \emph{not} \CM.

 Consider the filtration $\F = \{ \n^n \cap M \}_{n \geq 0}$ on $M$. We note that $\F$ is an $\n$-stable filtration on $M$. We have an exact sequence of $G(B)$-modules
$$ 0 \rt G_\F(M) \rt G(B) \rt G(A) \rt 0.$$
It follows that $G_\F(M)$ is \CM. Let $x$ be $B \oplus A$-superficial  and $M$-superficial \wrt \ $\F$.
We note that $e(M) = \ell(M/xM) = 1$. So $xM = \m M$ and $\mu(M) = 1$. We shift $\F$ to get a new filtration $\G = \{\G_n \}_{n \geq 0}$  of $M$ such that $\G_0 = M$ and $\G_1 \neq 1$.
We note that $G_\G(M) $ is \CM, say its $h$-polynomial is $a_0 + a_1 z + \cdots + a_s z^s$. Then as $G_\G(M)$ is \CM \  we get $a_i \geq 0$. Also $a_0 \neq 0$ by construction. As $e(G_\G(M)) = 1$ it follows that $a_0 = 1$ and $a_i = 0$ for $i > 0$. Thus $e_1(G_\G(M)) = 0$.
It follows that $G_\G(M) \cong k[Y]$. It follows that $G_\F(M) = k[Y](-a)$. Let $\alpha$ be the generator of $M$. Say image of $\alpha $ in $B$  is $\xi$. By going mod a $B \oplus A$ and $M$-superficial element \wrt \ $\F$ we obtain an exact sequence $0 \rt k \rt \ov{B} \rt \ov{A} \rt 0$. It follows that $\ov{\alpha}$ maps to $\delta$, a generator of $\soc(\ov{B})$. Say $\soc(\ov{B}) = \n^r$.  Thus $\xi = \delta + \theta x$ where $\theta \in B$. So $\ord_\n(\xi) \leq r$. So $a \leq r$.

If $J$ is a non-zero homogeneous ideal of $G(B)$ set $$i(J) = \min \{ \deg t \mid t \in J \ \text{where} \ t \ \text{is non-zero and homogeneous }\}.$$
We note that if $J \subseteq J^\prime$ then $i(J^\prime) \leq i(J)$. We note that $i(G_\F(M)) = a$.

Consider a primary decomposition of $G_\F(M)$ as an ideal in $G(B)$. As $G(A)$ is not \CM \ we have
$$ G_\F(M)  = Q_1\cap Q_2 \cap \cdots \cap Q_s \cap L, $$
where $L$ is $\M$-primary (here $\M$ is the unique homogeneous ideal of $G(B)$)  and $Q_i$ is $P_i$-primary where $P_i$ is some minimal prime of $G(B)$.
Set $Q = Q_1 \cap Q_2 \cap \cdots \cap Q_s$. Then $Q$ is a \CM \ ideal in $G(B)$ and $Q \supseteq G_\F(M)$.
We have an exact sequence
$$ 0 \rt G(B)/G_\F(M) \rt G(B)/Q \oplus G(B)/L \rt G(B)/(L + Q) \rt 0.$$
As $L$ is a zero-dimensional ideal of $G(B)$ we obtain $e(G(B)/Q) = e(G(A)) = e(B) -1$. It follows that $e(Q)$ (as a $G(B)$-module ) is one. We have an exact sequence of MCM $G(B)$-modules
$ 0 \rt Q \rt G(B) \rt G(B)/Q \rt 0$. Going mod a regular element $x^*$ of degree one we obtain a sequence $0 \rt Q/x^*Q \rt G(\ov{B}) \rt G(B)/(Q + x^*G(B)) \rt 0$.
As $e(Q) = 1$ it follows that $Q/x^*Q = k[-c]$ for some $c$. We note that the degree of socle of $G(\ov{B})$ is also $\ord_{\ov{B}}(\soc(\ov{B})) = r$. So $c = r$. Thus $Q = k[Y](-r)$. Thus $i(Q) = r$. As $Q \supseteq G_\F(M)$ we obtain $ r = i(Q) \leq i(G_\G(M)) = a \leq r$. So $i(G_\F(M)) = r$. Thus $G_\F(M) = k[Y](-r)$.

We now consider the exact sequence $ 0 \rt G_\F(M) \rt G(B) \rt G(A) \rt 0$. Going mod $x^*$ where $x$ is $B \oplus A $ and $M$-superficial \wrt \ $\F$ we obtain a sequence
$$ 0 \rt (0 \colon_{G(A)} x^*) \rt k[-r] \xrightarrow{f} G(\ov{B}) \rt G(A)/x^*G(A) \rt 0.$$
We note that  $f([\alpha]) = [\delta] \neq 0$. Thus $f$ is injective and hence $(0 \colon_{G(A)} x^*) = 0$. So $G(A)$ is \CM.

Now assume $\dim A = d \geq 2$. Let $ 0 \rt M \rt B \rt A \rt 0$. Let $\xb = x_1,\ldots, x_{d-1}$ be a $A\oplus B \oplus M$-superficial sequence.
So we have an exact sequence $ 0 \rt \ov{M} \rt \ov{B} \rt \ov{A} \rt 0$. We note $e(\ov{M}) = 1$ and $G(\ov{B})$ is Gorenstein. So by above argument $G(\ov{A})$ is \CM. By Sally descent, see \cite[2.2]{HM},  $G(A)$ is \CM.
\end{proof}

\section{Proof of Theorem \ref{main-strong}}
In this sections we give
\begin{proof}[Proof of Theorem \ref{main-strong}.]
We may assume that the residue field of $T$ is infinite. Otherwise consider the extension $T \rt T[X]_{\n T[X]}$. The hypotheses and conclusion of the theorem are invariant under this transformation.

Let $d = \dim A \geq 1$.
As $A$ is a Teter and a domain, by Theorem \ref{main} we have there exists a proper ideal $J$ of $A$ with $J \cong \omega_A$ and $\codim A/J \leq 1$. Furthermore $A/J = Q/(a)$ where $Q$ is a regular local ring of dimension $d$ and $a \neq 0$. By Theorem \ref{main}  we get $B = A\times_{Q/(a)}Q$ is the unique (up to isomorphism)  Teter Gorenstein approximation of $A$.

By \ref{strong} we may assume that  $G(A)$ is \CM. We note that by the exact sequence $0 \rt J \rt A \rt Q/(a)  \rt 0$ we obtain an exact sequence
$0 \rt G_\F(J) \rt G(A) \rt G(Q)/(a^*) \rt 0$. So $G_\F(J)$ is \CM.

Suppose $d \geq 2$. Then we choose $\xb = x_1, \ldots, x_{d-1}$ which is $Q \oplus B \oplus A$ and $\omega_A$ superficial \wrt \  $\F$. Set $\ov{(-)} = (-)\otimes_T T/(\xb)$. Then note $G(B)$ is Gorenstein if and only if $G(\ov{B})$ is Gorenstein. Furthermore $G_\F(\omega)$ is isomorphic to the canonical module of $G(A)$ (up to shift) if and only if $G_{\ov{\F}}(\ov{\omega})$ is isomorphic to the canonical module of $G(\ov{A})$ (up to shift). We also get $\ov{B} \cong \ov{A}\times_{Q/(a, \xb)} \ov{Q}$. However note $\ov{A}$ need not be a domain. But this is not pertinent to our result (the only requirement is that $\ov{B}$ is an appropriate fiber-product).
Thus we may assume that $d = 1$.
Consider the fiber-product exact diagram

\[
  \xymatrix
{
 0
 \ar@{->}[r]
  & Q
\ar@{->}[r]^{j}
\ar@{->}[d]^{g}
 & B
\ar@{->}[r]^{q}
\ar@{->}[d]^{p}
& A
\ar@{->}[r]
\ar@{->}[d]^{\pi}
&0
\\
 0
 \ar@{->}[r]
  & Q
\ar@{->}[r]^{j^\prime}
 & Q
\ar@{->}[r]^{\pi^\prime}
& Q/(a)
    \ar@{->}[r]
    &0
\
 }
\]
We note $g$ is an isomorphism.

Let $\n$ be the maximal ideal of $B$. As $Q$ is a DVR we have $\n Q = (x)$. We note $j^\prime $ is multiplication by $x^r$ for some $r \geq 1$.
On $Q$ (on the left hand side) we have two filtration's.   First set $\Hc = \{ j^{-1}(\n^n) \}_{n \geq 0}$ and $\Hc^\prime = \{ (j^\prime)^{-1}((x^n)) \}_{n \geq 0}$. Furthermore we have a natural injective map $\Hc_n \rt \Hc^\prime_n$ for all $n \geq 0$. It is evident that $\Hc^\prime = \{ (x)^{n-r} \}_{n\geq 0}$ (here $(x)^l = Q$ if $l \leq 0$).
Note that $\n  = \m\times_{(x)/(x^r)} (x)$. Let $n \geq r$.
We note that $(0, x^n) \in \n^n$. So by a diagram chase we obtain $\Hc_n \supseteq (x^{n-r})$ As $\Hc_n \subseteq \Hc^\prime_n = (x^{n-r})$. It follows that $\Hc_n = (x^{n-r})$ for all $n\geq r$.
For $0 \leq n \leq r -1$ we note that $\Hc_n \supseteq \Hc_r = Q$. So $\Hc_n = Q$ for $n \leq r.$ It follows that $\Hc \cong \Hc^\prime$.

We note that $G_\Hc(Q) = k[X](-r)$. We also have an exact sequence $0 \rt G_\Hc(Q) \rt G(B) \rt G(A) \rt 0$. As $G(A)$ is \CM \ we obtain that $G(B)$ is \CM.

We note that we have an exact sequence $ 0  \rt\omega_A \rt B \xrightarrow{p} Q \rt 0$. Set $\F^\prime = \{\omega_A \cap \n^n \}_{n \geq 0}$. We also have an exact sequence $0 \rt J \rt A \rt Q/(a) \rt 0$, with $J \cong \omega_A$ and the filtration $\F = \{ \m^n \cap J \}_{n\geq 0}$. We had that $G_\F(J)$ is \CM. We have a commutative diagram

\[
  \xymatrix
{
\
 \
  & \
 & 0
\ar@{->}[d]
& 0
\ar@{->}[d]
& \
\\
\
  & \
 & G_\Hc(Q)
\ar@{->}[r]^{G(g)}
\ar@{->}[d]^{G(j)}
& G_{\Hc^\prime}(Q)
\ar@{->}[d]^{G(j^\prime)}
&\
\\
 0
 \ar@{->}[r]
  & G_{\F^\prime}(\omega_A)
\ar@{->}[r]
\ar@{->}[d]^{\psi}
 & G(B)
\ar@{->}[r]^{G(p)}
\ar@{->}[d]^{G(q)}
& G(Q)
\ar@{->}[r]
\ar@{->}[d]^{G(\pi^\prime)}
&0
\\
 0
 \ar@{->}[r]
  &G_\F(J)
\ar@{->}[r]^{i^\prime}
 & G(A)
\ar@{->}[r]^{G(\pi)}
\ar@{->}[d]
& G(Q)/(a^*)
    \ar@{->}[r]
    \ar@{->}[d]
    &0
\\
\
\
  & \
\
 & 0
\
& 0
    \
    & \
\
 }
\]
We had earlier shown that $\Hc^\prime \cong \Hc$. So $G(g)$ is an isomorphism. It follows by Snake Lemma that $\psi$ is an isomorphism.
Thus we have an exact sequence $0 \rt G_\F(J) \rt G(B) \rt G(Q) \rt 1$. By going $x^*$ which is $G(B)$-regular and of degree one we obtain an exact sequence
$$ 0 \rt G_\F(J)/x^*G_\F(J) \rt G(\ov{B}) \rt k \rt 0.$$
It follows that $\soc(G_\F(J)/x^*G_\F(J)) = \soc(G(\ov{B}))$. We note that $G(B)$ is Gorenstein if and only if $\soc G(\ov{B}) = k$. We also have $G_\F(J)$ isomorphic to the canonical module of $G(A)$ upto shift if and only if  $\soc(G_\F(J)/x^*G_\F(J)) = k$. The result follows
\end{proof}

\section{Examples and proof of Theorem \ref{fin-rep}}
In this section we give several examples which illustrate our results. We also give a proof of Theorem \ref{fin-rep}. Eisenbud and Herzog gave a complete list of homogeneous \CM \ rings over an algebraically closed field $k$ of characteristic zero which are of finite representation type. Our proof of Theorem \ref{fin-rep} is by a case by case analysis of the examples of Eisenbud and Herzog.

We first show
\begin{proposition}
\label{dim-1-minmult} Let $R = \bigoplus_{n \geq 0}R_n$ be a standard algebra over an infinite field $k = R_0$. We assume $\dim R = 1$ and that $R$ is generically Gorenstein. Assume $R$ is \CM \ with minimal multiplicity and that $R$ is not Gorenstein. Let $\M$ be the maximal homogeneous ideal of $R$. Then $A = \wh{R_\M}$ is a strong Teter ring.
\end{proposition}
\begin{proof}
We note that $h_R(z) = 1 + hz$. We also have that $R$ is a level ring. So all generators of $\omega$ have the same degree.
We have a homogeneous inclusion $J = \omega_R(r) \rt  R$ with $J$ a proper ideal.  Choose $x \in R_1$ which is $R$-regular. We note that $R_n =xR_{n-1}$ for $n \geq 2$. If all generators of $J$ are of degree $\geq 2$ then $J = xJ^\prime$ for some ideal $J^\prime$ of $R$. We note that $J^\prime \cong J[+1] \cong \omega_R(r + 1)$. Thus we can assume that there exists an ideal $K$ generated in degree one with $K \cong \omega_R(s)$. We have $h_K(z) = hz + z^2$. So $h_{R/K}(z) = 1 + z$. Also $\dim R/K = 0$. It follows that $R/K$ is a graded  hypersurface ring. After completing we obtain that $A$ is a Teter ring by Theorem \ref{fiber}.

We note that $G(A) = A$ and $G(\wh{R/K}) = R/K$. Thus the exact sequence $0 \rt \omega_A \rt A \rt \wh{R/K} \rt 0$ induces an exact sequence
$0 \rt G_\F(\omega_A) \rt R \rt R/K \rt 0$. So $G_\F(\omega_A) \cong K$ is isomorphic to the canonical module of $G(A) = R$ up to shift.  It follows from \ref{main-strong} that $A$ is a strong Teter ring.
\end{proof}

\begin{example}
\label{d1-minmult} Let $R = k[X, Y, Z]/(XY, XZ, YZ)$. Then $R$ is of minimal multiplicity. Also it is not difficult to prove that $R$ is reduced. So by Theorem \ref{dim-1-minmult} we get that $A = \wh{R_{(X,Y,Z)}} = k[[X, Y, Z]]/(XY, XZ, YZ)$ is a strongly Teter ring.
\end{example}

\begin{example}\label{veron}
Let $k$ be a field. Set $S = k[X,Y]$ and let $R = S^{<m>}$ for some $m \geq 3$. Then $A = \wh{R}$ is strongly Teter.
\end{example}
\begin{remark}
  We note that $k[X, Y]^{<2>} = k[X^2, XY, Y^2]$ is a hypersurface ring and so Gorenstein.
\end{remark}
We now give
\begin{proof}[Proof of Example \ref{veron}]. We note that $\omega_S = XY S \subseteq S$. By \cite[3.1.3]{GW} we get that $\omega_R = \omega_S^{<m>}$. We note that $\omega_R \subseteq R$. We note that $\dim_k R_1 = m + 1$. Also note that $(\omega_R)_1 = (\omega_S)_m =  (X, Y)^{(m-2)}XY$. So $\dim_k (\omega_R)_1 = m -1$. Thus the maximal ideal of $ R/\omega_R$ is generated by two  elements and it is a \CM \ ring of dimension one. So $R/\omega_R$ is a hypersurface ring. After completing $A$ is a Teter ring by Theorem \ref{main}. By an argument similar to proof of Proposition \ref{dim-1-minmult} we get that $A$ is strongly Teter.
\end{proof}

\s Let H be an additive subsemigroup of $ \mathbb{N} = \{0, 1, 2, \cdots \}$ with $0 \in  H $ such that $\mathbb{N} \setminus H$ is finite. Let $R = K[H] $ be the
numerical semigroup ring of $H$. Set $A = \wh{R} = k[[H ]]$. We give a few examples of semigroup rings whose completions are Teter. The crucial observation is that
$$\omega_R = \sum_{\alpha \in \mathbb{N} \setminus H}Rt^{-\alpha}, \quad \text{see \cite[2.1.9]{GW}.}$$

\begin{example}\label{non-std-graded-1}
Consider $R = k[t^3, t^4, t^5]$. Then $A = \wh{R} = k[[t^3, t^4, t^5 ]]$ is a Teter ring.

\textit{Proof:} We note that $H = \{ 0, n, \text{where} \ n \geq 3\}.$  So $\omega_R = R{t^{-1}} + R{t^{-2}}$. Explicitly
$$ \omega_R = kt^{-2} + kt^{-1} + \sum_{n\geq 1}kt^{n}. $$
We note that $J = t^6 \omega_R \subseteq  R$  and note that $t^4, t^5 \in J$. Thus the maximal ideal of $R/J$ is generated by one element and $\dim R/J = 0$. Hence it is a hypersurface ring. So $A$ is a Teter ring by Theorem \ref{main}. \qed
\end{example}

\begin{example}\label{non-std-graded-2}
Consider $R = k[t^4, t^5, t^{11}]$.  Then $A = \wh{R} = k[[t^4, t^5, t^{11}]]$ is a Teter ring. However $A$ is \emph{not} strongly Teter.

\textit{Proof:} We note that $H = \{ 0, 4, 5, n, \text{where} \ n \geq 8\}$. So
$$ \omega_R = R{t^{-1}} + R{t^{-2}} + R{t^{-3}} + R{t^{-6}} + R{t^{-7}}. $$
Explicitly
$$ \omega_R = kt^{-7} + kt^{-6} + kt^{-3} + kt^{-2} + kt^{-1} + \sum_{n\geq 1}kt^{n}. $$
We note that $J = t^{11}\omega_R  \subseteq R$ and note that $t^4, t^5 \in J$. Thus the maximal ideal of $R/J$ is generated by one element and $\dim R/J = 0$. Hence  it is a hypersurface ring. So $A$ is a Teter ring by Theorem \ref{main}. By \cite[p.\ 20]{S}, $G(A)$ is NOT \CM. So by \ref{strong} we get that $A$ is NOT strongly Teter.  \qed
\end{example}

\s \emph{Integer matrix rings:} Let $A = \  ^t(a_1, \ldots, a_n)$ be an integer matrix and assume that the g.c.d of $\{a_1, \ldots, a_n \}$ is $1$. Define a grading of $S = k[x_1, \ldots, x_n]$ by defining $\deg x_i = a_i$. We have $S = \bigoplus_{n \in \Z}S_n$. Set $S_0 = R(A)$. We have that each $S_n$ is a $R(A)$-module.
Let $H$ is the sub-semigroup of $\mathbb{N}^n$ consisting of $\alpha  = (\alpha_1, \ldots, \alpha_n)$ with $\sum_{i = 1}^{n}a_i\alpha_i = 0$ then $R(A)$ is the semigroup ring $k[x^\alpha \mid \alpha \in H ]$ where $x^\alpha$ denotes $x^{\alpha_1}x_2^{\alpha_2}\cdots x_n^{\alpha_n}$. Note $R(A)$ is naturally an $\mathbb{N}$-graded ring. Set $\wh{R(A)} = k[[x^\alpha \mid \alpha \in H ]]$.
 We note that $R(A)$ and so $\wh{R(A)}$ are \CM \ rings, see \cite[16.1]{Y}.

 \begin{example}\label{U} Consider the integer matrix ring $A = \ ^t(2,-1,-1,-1)$. Then $\wh{R(A)}$ is a Teter ring.

 \begin{proof} By \cite[16.10]{Y} we have $$R(A) = k[x_1x_2^2, x_1x_2x_3, x_1x_3^2, x_1x_3x_4, x_1x_4^2, x_1x_2x_4 ]. $$
 We note that $\dim R(A) = 3$.
We also have  $S_1 = (x_1x_2, x_1x_3, x_1x_4)R(A)$ is isomorphic to the canonical module of $R(A)$ (up-to shift).
We have $x_2S_1 \subseteq R(A)$ and note that
$$x_2S_1 =  ( x_1x_2^2, x_1x_2x_3,  x_1x_2x_4 ).$$
So the graded maximal ideal of $R(A)/x_2S_1$ is generated by $3$ elements.  As \\  $\dim R(A)/x_2S_1$ is two and it is \CM \ it follows that $R(A)/x_2S_1$ is a hypersurface ring.
Taking completions it follows that $\wh{R(A)}$ is a Teter ring by Theorem \ref{main}.
 \end{proof}
 \end{example}

 \begin{example}\label{V} Consider the integer matrix ring $C = \ ^t(2,1,-1,-1)$. Then $\wh{R(C)}$ is a Teter ring.

 \begin{proof} By \cite[16.12]{Y} we have $$R(C) = k[x_1x_3^2,  x_1x_3x_4, x_1x_4^2, x_2x_3, x_2x_4]. $$
 We note that $\dim R(C) = 3$.
We also have  $S_{-1} = (x_3, x_4)R(C)$ is isomorphic to the canonical module of $R(A)$ (up-to shift).
We have $x_1x_3S_{-1} \subseteq R(A)$ and note that
$$x_1x_3S_{-1} =  ( x_1x_3^2, x_1x_3x_4).$$
So the graded maximal ideal of $D = R(A)/x_1x_3S_{-1}$ is generated by $3$ elements.  As \\  $ \dim D $ is two and it is \CM \ it follows that $D$ is a hypersurface ring.
Taking completions it follows that $\wh{R(C)}$ is a Teter ring by Theorem \ref{main}.
 \end{proof}
 \end{example}

 \s \label{EH-main} Let $k$ be an algebraically closed field of characteristic zero. Let $R$ be a standard graded \CM \ $k$-algebra of dimension $d$. Assume $R$ is \emph{not} Gorenstein and is of finite representation type. Then by \cite{EH}; $R$ is one of the following rings :

 $d =1$. Then $R$ is isomorphic to $k[X,Y, Z]/(XY, XZ, YZ)$.

 $d = 2$. Then $R$ is isomorphic to $k[X,Y]^{<m>}$ for some $m \geq 3$.

 $d = 3$. Then $R$ is isomorphic to either $U = k[X, Y, Z]^{<2>}$ or $R$ is isomorphic to $V = k[X_1, X_2, X_3, X_4, X_5]/(X_1^2 -X_2X_3, X_2X_5 - X_1X_4, X_1X_5 -X_3X_4)$.

 We now give
 \begin{proof}[Proof of Theorem \ref{fin-rep}]
 We prove case by case. Set $d = \dim R$. When $d = 1$ we note that $R$ is isomorphic to $k[X,Y, Z]/(XY, XZ, YZ)$. So by \ref{d1-minmult} we get $\wh{R_\M}$ is a Teter ring.
 When $d = 2$ we have $R$ is isomorphic to $k[X,Y]^{<m>}$ for some $m \geq 3$. So by \ref{veron} we get that $\wh{R_\M}$ is a Teter ring. When
 $d = 3$ we note that $\wh{R( \ ^t(2,-1,-1,-1))}$ is isomorphic to the completion of $U$, see \cite[16.10]{Y}. So by \ref{U} we get that $\wh{U_\M}$ is Teter. We also have
 $\wh{R( \  ^t(2,1,-1,-1))}$ is isomorphic to the completion of $V$, see \cite[16.12]{Y}. So by \ref{V} we get that $\wh{V_\M}$ is Teter. The result follows.
 \end{proof}

 \begin{remark}
   Let $k$ be a field of characteristic zero. Consider $A= k[[t^3, t^4, t^5]]$. In \ref{non-std-graded-1} we proved that $A$ is Teter. Note $A$ bi-rationally dominates $k[[t^3, t^4]]$ (which is $E_6$). So $A$ has finite representation type. Thus there exists examples  of completions of non-homogeneous rings of finite representation type which are Teter.
 \end{remark}

 \begin{remark}
   We note that all the examples in \ref{EH-main} have minimal multiplicity. Consider $A = k[[t^4, t^5, t^{11}]]$  which is a Teter ring, see \ref{non-std-graded-2}. We note $A$ does not have minimal multiplicity.
 \end{remark}

 \begin{example}
 If $(A, \m)$ is a Teter ring then it is easily verified that both \\ $A[X_1,\ldots, X_n]_\n$ (where  $\n = (\m, X_1, \ldots, X_n)$) and $A[[X_1, \ldots, X_n]]$ are Teter rings.
 Applying this result to our examples creates  bountiful examples of Teter rings which are domains.
 \end{example}

It is perhaps pertinent to give an example of a local domain which is not Teter.  To enable this we prove:
\begin{proposition}\label{type-bound}
Let $(A,\m)$ be a \CM \ local ring of dimension $d$. If $A$ is Teter then $\type(A) = \mu(\m) - d$.
\end{proposition}
\begin{remark}
In the above proposition $\type(A)$ denotes the \CM \ type of $A$, i.e., the $k =A/\m$ dimension of $\Ext^d_A(k, A)$.
\end{remark}
We now give
\begin{proof}[Proof of Proposition \ref{type-bound}.]
We may assume that the residue field of $A$ is infinite. Let $B$ be a Gorenstein Teter approximation of $A$. We have an exact sequence $0 \rt M \rt B \rt A \rt 0$; where $e(M) = 1$. We go mod $\xb$ an $M \oplus B \oplus A$-superficial sequence. So we have
$0 \rt k \rt \ov{B} \rt \ov{A} \rt 0$. So $\ov{A}$ is  Teter.  We note that $\type(A) = \type(\ov{A})$ and $\mu(\ov{\m}) = \mu(\m) -d$. Thus we can assume $A$ is an Artin Teter ring.

Dualizing the sequence $0 \rt k \rt B \rt A \rt 0$ we obtain an exact sequence $0 \rt \omega_A \rt B \rt k \rt 0$. So we obtain a surjective map $\omega_A \rt \m$ with kernel $k$.
Thus we have exact sequence $0 \rt k \rt \omega_A \rt \m \rt 0$. We apply the functor $\Hom_A(k, -)$ and noting that $\omega_A$ is the canonical module of $A$ we obtain
$\Hom_A(k, \m) \cong \Ext^1_A(k,k)$. Thus $\type(A) = \mu(\m)$.
\end{proof}

\begin{example}
Consider the following example from \cite[p.\ 395]{S2}. Let \\ $A = k[[t^5, t^6, t^7, t^9 ]]$.
We have $\mu(\m) = 4$. But $\type(A) = 2$. So $A$ is not Teter by \ref{type-bound}.
\end{example}

\end{document}